\documentclass{amsart}
\usepackage{amssymb,amsmath,amsthm}
\usepackage{mathbbol}
\usepackage{txfonts}

\numberwithin{equation}{section}

\newtheorem{thm}{Theorem}[section]
\newtheorem{lem}[thm]{Lemma}
\newtheorem{prop}[thm]{Proposition}
\newtheorem{cor}[thm]{Corollary}

\theoremstyle{remark}
 \newtheorem{rmk}[thm]{Remark}
 \newtheorem{ex}[thm]{Example}

\newcommand{\R}{\mathbb{R}}
\newcommand{\N}{\mathbb{N}}

\newcommand{\sgn}{\operatorname{sgn}}
\newcommand{\lip}{\operatorname{Lip}}
\newcommand{\FV}{\operatorname{FV}}
\newcommand{\fillvol}{\operatorname{Fillvol}}
\newcommand{\vol}{\operatorname{Vol}}
\newcommand{\dist}{\operatorname{dist}}
\newcommand{\spa}{\operatorname{span}}
\newcommand{\mass}{\mathbf{M}}
\newcommand{\p}{\partial}

\begin{document}
 
\title{Lipschitz non-extension theorems into jet space Carnot groups}

\author{S\'everine Rigot}

\address{Laboratoire J.-A. Dieudonn\'e, CNRS-UMR 6621, Universit\'e de Nice Sophia-Antipolis, Parc Valrose, 06108 Nice cedex 02, France}
\email{rigot@unice.fr}

\author{Stefan Wenger}

\address{Department of Mathematics, University of Illinois at Chicago, 851 S. Morgan Street, Chicago, IL 60607--7045, USA }
\email{wenger@math.uic.edu}

\thanks{The second author is partially supported by NSF grant DMS 0707009}

\begin{abstract} We prove non-extendability results for Lipschitz maps with target space being jet spaces equipped with a left-invariant Riemannian distance, as well as jet spaces equipped with a left-invariant sub-Riemannian Carnot-Carath\'eodory distance. The jet spaces give a model for a certain class of Carnot groups, including in particular all Heisenberg groups.
\end{abstract}

\maketitle

\section{Introduction}

The problem of the extendability of partially defined Lipschitz maps between metric spaces is a classical problem, which has received constant research interest over the years. Beyond the well-known Kirszbraun's Theorem \cite{k,mcs}, most of the first classical results concern spaces with a linear structure, see e.g. \cite{V1,jl,jls}. Other geometrical contexts have been investigated as well, see e.g. \cite{V2,ls,lps,lschli,ln} and the references therein. Concerning Carnot groups, the case of the $n$-th Heisenberg group as target space has recently been considered, see e.g. \cite{Balogh-Faessler,gromov,Magnani_preprint}.

 In the present paper, we exhibit Lipschitz non-extendability  results for maps to jet spaces, equipped either with a left-invariant Riemannian or a left-invariant sub-Riemannian Carnot-Carath\'eodory distance, respectively. 
The jet spaces $J^k(\R^n)$, see Section~\ref{section:prelims}, give a model for a certain class of $(k+1)$-step Carnot groups, including the $n$-th Heisenberg group $J^1(\R^n)$, the Engel group $J^2(\R)$ and more generally the model filiform groups $J^k(\R)$. Recall that a Carnot group is a connected, simply connected Lie group with stratified Lie algebra. It can be endowed with two kinds of natural distances: the left-invariant Riemannian distance $d_0$ coming from an inner product on the Lie algebra which makes the layers of the stratification orthogonal and the left-invariant sub-Riemannian Carnot-Carath\'eodory distance $d_c$ defined with respect to the first, so-called horizontal, layer of the stratification. One has the relationship $d_0\leq d_c$, these two distances are however not bi-Lipschitz equivalent.

We now turn to the Lipschitz extension problem for maps to $(J^k(\R^n), d_0)$ and first note that every Lipschitz map from a bounded subset of $\R^m$, $m\geq 1$ arbitrary, to $(J^k(\R^n), d_0)$ possesses a Lipschitz extension to all of $\R^m$. This follows from the fact that the Lie exponential map is a global diffeomorphism and from the Lipschitz extension property of $\R^n$. Our first theorem is:

\begin{thm}\label{thm:non-ext-d0}
 For every $n, k\geq 1$, there is an (unbounded) subset $A\subset\R^{n+1}$ and a Lipschitz map $f: A\to(J^k(\R^n), d_0)$ which does not admit any Lipschitz extension $\bar{f}: \R^{n+1}\to(J^k(\R^n), d_0)$. 
\end{thm}

In contrast, when $J^k(\R^n)$ is equipped with the distance $d_c$, there are Lipschitz maps from the unit $n$-sphere $S^n\subset\R^{n+1}$ which do not admit Lipschitz extensions to the closed unit ball $\bar{D}^{n+1}$.

\begin{thm}\label{thm:main-intro}
 For every $n, k\geq 1$, there is a Lipschitz map $f: S^n\to (J^k(\R^n), d_c)$ which does not admit any Lipschitz extension $\bar{f}: \bar{D}^{n+1}\to (J^k(\R^n), d_c)$.
\end{thm}
Both theorems come as a consequence of a more general construction, given in Section~\ref{section:main-results-proof}, which may be considered to be the core of this paper. Roughly speaking, we construct a map $F: S^n\to J^k(\R^n)$ which is Lipschitz with respect to the distance $d_c$, and thus also to $d_0$, such that any extension to $\bar{D}^{n+1}$ of its dilate $\delta_L\circ F$ must have Lipschitz constant with respect to $d_0$ of order at least $L^{1+\frac{k}{n+1}}$. Since $d_c$ is $1$-homogeneous with respect to the dilations, $\delta_L\circ F$ has Lipschitz constant of order $L$ with respect to $d_c$, and thus also to $d_0$. We use this super-linear growth of the Lipschitz constant to prove our Lipschitz non-extension results stated above. We finally mention that we actually construct a whole family of such maps $F$, very naturally obtained through the jet map construction.

Theorem~\ref{thm:main-intro} extends an analogous result for the Heisenberg groups $J^1(\R^n)$, recently obtained by Balogh-F\"assler in  \cite{Balogh-Faessler} with different methods. The proof in \cite{Balogh-Faessler} relies heavily on the fact that $J^1(\R^n)$ is purely $(n+1)$-unrectifiable. We note that arguments relying on pure unrectifiability as in \cite{Balogh-Faessler} could not be invoked to obtain Theorem~\ref{thm:main-intro} in our generality. Indeed, $(J^k(\R^n), d_c)$ is purely $m$-unrectifiable if and only if 
$$m> \binom{n+k-1}{k},$$
see \cite{Magnani}, but $\binom{n+k-1}{k}\geq n+1$ as soon as $n, k\geq 2$.

As a consequence of our constructions, we also obtain the following Lipschitz non-extension result for maps between jet spaces. In order to state the result, we introduce some notation. A map $f: A\subset J^l(\R^m)\to J^k(\R^n)$ will be called $(d, d')$-Lipschitz if $f$ is Lipschitz as a map from $A\subset (J^l(\R^m), d)$ to $(J^k(\R^n), d')$, where $d$ and $d'$ denote either $d_0$ or $d_c$. 

\begin{cor}\label{cor:jet-jet}
 Let $l,m,n\geq 1$ be such that $\binom{m+l-1}{l}\geq n+1$, and let $k\geq 1$. Then there exists a subset $A\subset J^l(\R^m)$ and a map $f: A\to J^k(\R^n)$ which is $(d_0, d_c)$-Lipschitz but such that no extension $\bar{f}: J^l(\R^m)\to J^k(\R^n)$ of $f$ can be $(d_c, d_0)$-Lipschitz.
\end{cor}

Since $d_0\leq d_c$, the notion of $(d_0, d_c)$-Lipschitz map is the strongest one and that of $(d_c, d_0)$-Lipschitz map the weakest. In particular the map $f$ in Corollary~\ref{cor:jet-jet} gives at the same time an example of a partially defined $(d_0, d_0)$-Lipschitz, resp. $(d_c, d_c)$-Lipschitz, map which does not admit any $(d_0, d_0)$-Lipschitz, resp. $(d_c, d_c)$-Lipschitz, extension.

We finally note that our construction in Section~\ref{section:main-results-proof} can be used to recover lower bounds for the $(n+1)$-st filling volume function $\FV_{n+1}$ on $(J^k(\R^n), d_0)$. This will be investigated in Section~\ref{section:lower-bounds-FV}. Such bounds have recently been obtained by Young in \cite{Young-nilpotent-isop} with different methods. One specific feature of our construction is that we can easily exhibit explicit constants. 

The paper is organized as follows. In Section~\ref{section:prelims} we recall the definition of jet spaces together with the metric structures we are interested in. In Section~\ref{section:main-results-proof} we give our main construction and show how the results stated above follow. The final Section~\ref{section:lower-bounds-FV} applies the methods developed in Section~\ref{section:main-results-proof} to obtain lower bounds for the filling volume functions in jet spaces.

\section{Preliminaries}\label{section:prelims}

\subsection*{Jet spaces as Carnot groups}

The jet spaces $J^k(\R^n)$ give a model for a certain class of Carnot groups. This class includes a model for the Heisenberg group $J^1(\R^n)$, see Example \ref{heisenberg}, for the Engel group $J^2(\R)$ and more generally for the model filiform groups $J^k(\R)$. We follow here \cite{Warhurst} to which we refer for all  details. 

Let $n, k\geq 1$. The $k$-th order Taylor polynomial of a $C^k$-smooth function $f:\R^n \rightarrow \R$ at $x_0$ is given by
\begin{equation*}
 T^k_{x_0}(f) (x) = \sum_{j=0}^k \sum_{I\in I(j)} \p_I f(x_0) \, \dfrac{(x-x_0)^I}{I!}
\end{equation*}
where, for $j = 0,\dots,k$, each $j$-index $I = (i_1,\dots,i_n)$ satisfies $|I| = i_1 + \cdots + i_n = j$, $I(j)$ denotes the set of all $j$-indices, $I!=i_1!\dots i_n!$,  $x^I=(x_1)^{i_1}\cdots (x_n)^{i_n}$, and 
\begin{equation*}
 \p_I f(x_0) = \dfrac{\p^j f}{\p x_1^{i_1}\dots \p x_n^{i_n}} (x_0).
\end{equation*}
Two functions $f_1$, $f_2 \in C^k(\R^n,\R)$ are said to be equivalent at $x_0$, $f_1 \sim_{x_0} f_2$, if and only if $T^k_{x_0}(f_1) =  T^k_{x_0}(f_2)$. The equivalence class of $f$ at $x_0$ is denoted by $j_{x_0}^k(f)$. The $k$-jet space over $\R^n$ is defined as
\begin{equation*}
 J^k(\R^n) = \bigcup_{x_0\in\R^n} \, C^k(\R^n,\R)/\sim_{x_0}.
\end{equation*}
 We identify 
\begin{equation*}
 J^k(\R^n) \thickapprox \R^n \times \R^{d_k^n} \times \dots \times \R^{d_0^n},
\end{equation*}
where $d_j^n = \binom{n+j-1}{j}$, via the global coordinates $(x,u^k,\dots,u^0)$, where $x(j_{x_0}^k(f)) = x_0$ and $u^j = (u_I^j)_{I\in I(j)}$ with $u_I^j(j_{x_0}^k(f)) = \p_I f(x_0)$ for $j=0,\dots,k$.

The horizontal subbundle of the tangent bundle is defined by
\begin{equation*}
\mathcal{H}_0 = \spa \left\{X_i; \, i=1,\dots,n\right\} \oplus \spa \left\{\p_{u_I^k};\, I \in I(k)\right\}
\end{equation*}
where, setting $e_i=(\delta_1^i,\dots,\delta_n^i) \in I(1)$,
\begin{equation*}
X_i = \p_{x_i} + \sum_{j=0}^{k-1} \sum_{I\in I(j)} u_{I+e_i}^{j+1} \p_{u_I^j}.
\end{equation*}
For $j=1,\dots,k$, we set
\begin{equation*}
\mathcal{H}_j = \spa \left\{\p_{u_I^{k-j}};\, I \in I(k-j)\right\}.
\end{equation*}
The only non trivial commutators are 
\begin{equation*}
 \left[\p_{u_{I+e_i}^{j+1}}, X_i\right] = \p_{u_I^j}, \quad I\in I(j), \quad j=0,\dots,k-1.
\end{equation*}
It follows that $\mathcal{H}_j = [\mathcal{H}_0,\mathcal{H}_{j-1}]$ where $j=1,\dots,k$ and $\mathcal{H} = \mathcal{H}_0 \oplus \dots \oplus \mathcal{H}_k$ is a $(k+1)$-step stratified Lie algebra which span $TJ^k(\R^n)$ pointwise. Corresponding to this Lie algebra, there is a Carnot group $G^k_n$, i.e., a connected, simply connected and stratified Lie group, unique up to isomorphism. One can actually determine a product $\odot$ on $J^k(\R^n)$ that makes $(J^k(\R^n),\odot)$ isomorphic to $G^k_n$ and we consider in the present paper $J^k(\R^n)$ as being equipped with this product and structure of Carnot group with Lie algebra $\mathcal{H}$. We will however not need here an explicit expression for $\odot$ and we refer to \cite{Warhurst} for further details. 

A Carnot group is naturally equipped with a family of dilations. For $L>0$, they are given on the Lie algebra by $\delta_L(\sum_{j=0}^k V_j) =  \sum_{j=0}^k L^{j+1} V_j$, where $V_j \in \mathcal{H}_j$, and define in turn dilations on the group via the exponential map. In the case of $J^k(\R^n)$, we have
\begin{equation*}
 \delta_L(x,u^k,u^{k-1},\dots,u^0) = (Lx,Lu^k,L^2 u^{k-1}, \dots,L^{k+1} u^0).
\end{equation*}
These dilations are group homomorphisms.

\begin{ex} \label{heisenberg} The Heisenberg group. When $k=1$, $n\geq 1$, the Lie algebra $\mathcal{H}$ is isomorphic to the Lie algebra of the $n$-th Heisenberg group. Using normal coordinates of the second kind, one recovers the above introduced coordinates on $J^1(\R^n)$. Relabelling some of the coordinates, namely, $y_i = u_{e_i}^1$, $y = (y_1,\dots,y_n)$ and $z=u^0$, the group law is then given by 
\begin{equation*}
 (x,y,z) \odot (x',y',z) = (x+x',y+y',z+z'+\sum_{i=1}^n y_i x'_i),
\end{equation*}
the left invariant horizontal vector fields being
\begin{equation*}
 X_i = \p_{x_i} + y_i \p_z \quad \text{and} \quad Y_i = \p_{y_i},
\end{equation*}
with the only non trivial commutators $[Y_i,X_i] = \p_z$. The dilations are 
\begin{equation*}
 \delta_L(x,y,z) = (Lx,Ly,L^2 z).
\end{equation*}
\end{ex}

\subsection*{Metric structures}
We equip $J^k(\R^n)$ with the left invariant Riemannian metric $g_0$ which makes $(X_1,\dots,X_n,\p_{u_I^j})_{j=0,\dots,k, I\in I(j)}$ an orthonormal basis and we denote by $d_0$ the induced left invariant Riemannian distance.

The Carnot-Carath\'eodory distance $d_c$ on $J^k(\R^n)$ is the sub-Riemannian distance defined by 
\begin{equation*}
 d_c(x,y) = \inf \{ length_{g_0} (\gamma); \; \gamma \text{ horizontal } C^1 \text{ curve joining } x \text{ to } y\},
\end{equation*}
where a $C^1$ curve is said to be horizontal if, at every point, its tangent vector belongs to the horizontal subbundle of the tangent bundle. Important properties of the Carnot-Carath\'eodory distance are that it is left invariant and 1-homogeneous with respect to the dilations, i.e., $d_c(\delta_L(x),\delta_L(y)) = L\,d_c(x,y)$ for all $x$, $y \in J^k(\R^n)$ and all $L\geq 0$. This follows from the homogeneity of the horizontal vector fields with respect to the dilations. 

We obviously have the relationship
\begin{equation*}
 d_0 \leq d_c,
\end{equation*}
which will be crucial in the present paper. However it is well-known that $d_0$ and $d_c$ are not bi-Lipschitz equivalent. In particular $J^k(\R^n)$-valued maps that are Lipschitz with respect to the distance $d_c$ are also Lipschitz with respect to the distance $d_0$, but the converse is false in general. We will thus always specify the distance $J^k(\R^n)$ is endowed with when speaking about $J^k(\R^n)$-valued Lipschitz maps, writing explicitly $(J^k(\R^n),d_0)$, resp. $(J^k(\R^n),d_c)$, as target space.

\subsection*{Jet maps}
A differentiable map from $\R^n$ with values in $J^k(\R^n)$ is said to be horizontal if the image of its differential lies is the horizontal subbundle of the tangent bundle. If $f\in C^{k+1}(\R^n,\R)$, then the map $j^k(f):x \mapsto j_{x}^k(f)$ is a horizontal $C^1$ map with 
\begin{equation*}
 \p_{x_i} (j^k(f)) (x) =  X_i(j_x^k(f)) + \sum_{I \in I(k)} \p_{I+e_i} f(x) \,\p_{u_I^k}
\end{equation*}
for $i=1,\dots,n$. Fix $x,y\in\R^n$ and define $\gamma(t) := (1-t)x+ty$. It follows that $t\in [0,1] \mapsto j^k(f)(\gamma(t))$ is a horizontal curve between $x$ and $y$ and we get from the very definition of $d_c$ that 
\begin{equation*}
 d_c(j^k(f)(x),j^k(f)(y)) \leq \sup_{t\in [0,1]} \left(1+\sum_{I\in I(k)} \sum_{i=1}^n (\p_{I+e_i} f(\gamma(t))^2\right)^{1/2} \, \|y-x\|.
\end{equation*}
In particular, the map $j^k(f):\R^n \rightarrow (J^k(\R^n),d_c)$ is locally Lipschitz. 

\section{Proof of the main results}\label{section:main-results-proof}

As already stressed in the introduction all results in the present paper rely on a single construction which in turn may be considered to be the core of our paper. We give this construction below and next show how Theorem~\ref{thm:non-ext-d0}, Theorem~\ref{thm:main-intro}, and Corollary~\ref{cor:jet-jet} follow. 

\subsection*{The main construction}
Let $n,k\geq 1$ and set $Q^n:= [0,1]^n$. With each pair $f_0, f_1: \R^n\to\R$ of $C^{k+1}$-smooth functions satisfying 
\begin{equation}\label{eqn:f0=f1}
 \p_I f_0(x) = \p_I f_1(x) 
\end{equation}
for all $x\in \partial Q^n$ and all $I\in I(j)$, $j=0,\dots, k$, we associate a Lipschitz map $$F:\partial Q^{n+1}\to (J^k(\R^n), d_c)$$ as follows. Denoting points in $\partial Q^{n+1} = (\partial Q^n \times [0,1]) \cup (Q^n \times \{0,1\})$ by $(x,t)$, we set
\begin{equation}\label{eqn:def-Lip-assoc}
 F(x,t):= \left\{
  \begin{array}{rl}
   j^k(f_0)(x) & \text{if $x\in Q^n$ and $t=0$}\\
   j^k(f_1)(x) & \text{if $x\in Q^n$ and $t=1$}\\
   j^k(f_0)(x) & \text{if $x\in\partial Q^n$ and $t\in(0,1)$.}\\ 
  \end{array}\right.
\end{equation}
Since $j^k(f_0) = j^k(f_1)$ on $\partial Q^n$, it is clear that $F$ is well-defined.
Furthermore, since $j^k(f_0)$, $j^k(f_1): Q^n \rightarrow (J^k(\R^n), d_c)$ are Lipschitz maps, see Section~\ref{section:prelims}, it follows that $F:\partial Q^{n+1}\to (J^k(\R^n), d_c)$ is also Lipschitz. 

Next, by homogeneity of the Carnot-Carath\'eodory distance with respect to the dilations, see Section~\ref{section:prelims}, we have that $\delta_L\circ F: \partial Q^{n+1} \rightarrow (J^k(\R^n), d_c)$ is $(L\lip_{d_c}(F))$-Lipschitz for every $L\geq 0$. Now, denote by $\iota: (J^k(\R^n), d_c) \to (J^k(\R^n), d_0)$ the identity map. This map is $1$-Lipschitz and it follows that the map $\iota\circ\delta_L\circ F: \partial Q^{n+1} \rightarrow (J^k(\R^n), d_0)$ is $(L\lip_{d_c}(F))$-Lipschitz for every $L\geq 0$.

\begin{prop}\label{thm:general-case}
 Let $f_0, f_1$ be as above. Then for every $L> 0$, any Lipschitz extension $\bar{F}: Q^{n+1}\to (J^k(\R^n), d_0)$ of $\iota\circ\delta_L\circ F$ has Lipschitz constant 
\begin{equation}\label{eqn:Lip-const-ext}
 \lip_{d_0} (\bar{F}) \geq L^{1+\frac{k}{n+1}}\left| \int_{Q^n}(f_0-f_1)dx\right|^{\frac{1}{n+1}}.
\end{equation}
\end{prop}

The proof of Proposition~\ref{thm:general-case} can be thought of as some variant of calibration's techniques. Similar techniques have already been used in \cite[Chapter 8]{epstein} and later by Burillo \cite{b} to obtain lower bounds for the filling volume of suitable $n$-cycles in the $n$-th Heisenberg group. The construction of these $n$-cycles in \cite{b} is rather complicated. The construction of our map $F$ is in contrast fairly simple and moreover very natural when using the jet space model. We first prove two auxiliary lemmas.

 \begin{lem}\label{lem:diff-form-norm}
  The $(n+1)$-form $\omega: = dx_1\wedge\dots\wedge dx_n\wedge du^0$ on $J^k(\R^n)$ satisfies
\begin{equation}\label{eqn-omega}
|\omega(p)(V_1,\dots, V_{n+1})|\leq 1
\end{equation}
for all $p\in J^k(\R^n)$ and all $V_1,\dots, V_{n+1}\in T_pJ^k(\R^n)$ with $\|V_i\|_{g_0}\leq 1$. 
 \end{lem}

\begin{proof}
 Let $p=(x,u^k,\dots,u^0)$ and 
\begin{equation*}
 V_i = \sum_{l=1}^n v_l^i \, X_l(p) + \sum_{j=1}^k \sum_{I\in I(j)} v_I^i \, \p_{u_I^j} + v_0^i \, \p_{u^0}
\end{equation*}
be as in the statement. We have
\begin{equation*}
\begin{split}
 \omega(p)(V_1,\dots, V_{n+1}) &= \sum_{\sigma \in \mathcal{S}_{n+1}} \sgn(\sigma) \, v_1^{\sigma(1)} \cdots v_n^{\sigma(n)} \, (v_1^{\sigma(n+1)} u_{e_1}^1 + \dots  +v_n^{\sigma(n+1)} u_{e_n}^1 + v_0^{\sigma(n+1)})\\
&= \det (\overline V_1,\dots,\overline V_{n+1}) + \sum_{l=1}^n u_{e_l}^1 \, dx_1\wedge\dots\wedge dx_n\wedge dx_l \,(V_1,\dots,V_{n+1})\\
&=\det (\overline V_1,\dots,\overline V_{n+1}) 
\end{split}
\end{equation*}
where $\overline V_i = ( v_1^i, \dots ,v_n^i,v_0^i) \in \R^{n+1}$. Since $\|\overline V_i\| \leq \|V_i\|_{g_0}\leq 1$, where $\|\overline V_i\|$ denotes the standard Euclidean norm of $\overline V_i$ in $\R^{n+1}$, inequality \eqref{eqn-omega} follows.
\end{proof}

\begin{lem}\label{lem:stokes}
 Let $h_1, \dots, h_{n+1}:Q^{n+1}\to\R$ be Lipschitz functions. Then 
 \begin{equation*}
  \int_{Q^{n+1}} dh_1\wedge\dots\wedge dh_{n+1} = \int_{\partial Q^{n+1}} h_1dh_2\wedge\dots\wedge dh_{n+1}.
 \end{equation*}
\end{lem}

Here, similarly to the smooth case, the exterior derivative of $h_i$ is a.e.-defined to be the $1$-form $dh_i = \sum_{j=1}^{n+1} \partial_{x_j} h_i \, dx_j$. This gives more explicitly: 
\begin{equation*}
 \int_{Q^{n+1}} dh_1\wedge\dots\wedge dh_{n+1} = \int_{Q^{n+1}}\det\left(\partial_{x_j} h_i\right) \, dx_1\cdots dx_{n+1}
\end{equation*}
and 
\begin{multline*}
 \int_{\partial Q^{n+1}} h_1dh_2\wedge\dots\wedge dh_{n+1} \\ = \sum_{j=1}^{n+1} \int_{[0,1]^n} \hat h_1^{j,1} \det\left(\partial_{x_k} \hat h_i^{j,1}\right)_{\substack{i\geq 2 \\ k\not=j}}  \, d\hat x_j
- \int_{[0,1]^n} \hat h_1^{j,0} \det\left(\partial_{x_k} \hat h_i^{j,0}\right)_{\substack{i\geq 2 \\ k\not=j}} \, d\hat x_j
\end{multline*}
where $\hat x_j = (x_1,\dots,x_{j-1},x_{j+1},\dots,x_{n+1}) \in \R^n$, $ \hat h_i^{j,l}(\hat x_j ) = h_i (x_1,\dots,x_{j-1},l,x_{j+1},\dots,x_{n+1})$ for $l=0,1$.

\begin{proof}
 If $h_1, \dots, h_{n+1}$ are $C^2$-smooth, the equality follows from integration by parts. The general case follows by a smoothing argument together with the weak$^*$ continuity of determinants, see e.g.~Theorem 2.16 in \cite{Ambrosio-Fusco-Pallara}.
\end{proof}

\begin{proof}[Proof of Proposition~\ref{thm:general-case}]
Fix $L>0$ and let $\bar{F}: Q^{n+1}\to (J^k(\R^n), d_0)$ be a Lipschitz extension of $\iota\circ\delta_L\circ F$. Let $(x, u^k, \dots, u^0)$ be the global coordinates introduced in Section~\ref{section:prelims}, and let $h_i$ denote the $x_i$-coordinate and $h_{n+1}$ the $u^0$-coordinate of $\bar{F}$, thus
$\bar{F} = (h_1, \dots, h_n, h, h_{n+1})$ for some vector-valued function $h$.
For $(x,t)\in \partial Q^{n+1}$, we then have $h_i(x, t) = Lx_i$ if $i=1, \dots, n$, and 
\begin{equation*}
h_{n+1}(x, t)=\left\{
  \begin{array}{rl}
   L^{k+1}f_0(x) & \text{if $x\in Q^n$ and $t=0$}\\
    L^{k+1}f_1(x) & \text{if $x\in Q^n$ and $t=1$}\\
    L^{k+1}f_0(x) & \text{if $x\in\partial Q^n$ and $t\in(0,1)$.}\\ 
  \end{array}\right.
\end{equation*}
It follows that
\begin{equation}\label{eqn:boundary-int}
 \begin{split}
 \int_{\partial Q^{n+1}} h_1dh_2\wedge\dots\wedge dh_{n+1} &= (-1)^{n} L^{n+k+1} \int_{Q^n} x_1\partial_{x_1} (f_1-f_0) \, dx\\
   &= (-1)^n L^{n+k+1} \int_{Q^n}(f_0-f_1)dx,
 \end{split}
\end{equation}
where we used integration by parts and \eqref{eqn:f0=f1} to obtain the last equality.
Now let $\omega:= dx_1\wedge\dots\wedge dx_n\wedge du^0$ be the $(n+1)$-form $\omega$ on $J^k(\R^n)$ from Lemma~\ref{lem:diff-form-norm}. We have 
\begin{equation*}
 \bar{F}^*\omega = dh_1\wedge\dots\wedge dh_{n+1}
\end{equation*}
and hence, with \eqref{eqn-omega},
\begin{equation}\label{eqn:int-Q}
\left| \int_{Q^{n+1}} dh_1\wedge\dots\wedge dh_{n+1}\right| =  \left| \int_{Q^{n+1}} \bar{F}^*\omega\right| \leq \lip_{d_0}(\bar{F})^{n+1}.
\end{equation}
Now, \eqref{eqn:Lip-const-ext} clearly follows from Lemma~\ref{lem:stokes} together with \eqref{eqn:boundary-int} and \eqref{eqn:int-Q}. This completes the proof.
\end{proof}

\subsection*{Non-extendability of certain Lipschitz maps}

We now show how Theorem~\ref{thm:non-ext-d0}, Theorem~\ref{thm:main-intro} and Corollary~\ref{cor:jet-jet} can be deduced from the preceding construction.

\begin{proof}[Proof of Theorem~\ref{thm:non-ext-d0}]
  Choose any $C^{k+1}$-smooth functions  $f_0, f_1: \R^n\to\R$ satisfying \eqref{eqn:f0=f1} and
 \begin{equation*}
  \int_{Q^n} f_0dx \not=  \int_{Q^n} f_1dx.
 \end{equation*}
Let $F:\partial Q^{n+1}\to (J^k(\R^n), d_c)$ be the Lipschitz map associated with $f_0$, $f_1$ by \eqref{eqn:def-Lip-assoc}. For $L\geq 0$, let $A_L:=[- 2^{L-1}, 2^{L-1}]^{n+1} \subset \R^{n+1}$ denote the cube centered at the origin with edge-length $2^L$. Set 
 \begin{equation*}
  A:= \bigcup_{L=0}^\infty \partial A_L,
 \end{equation*} 
and define a map $f: A\to J^k(\R^n)$ by
 \begin{equation*}
  f\left(2^L(z-e)\right):= \delta_{2^L}\circ F(z)
 \end{equation*}
 for all $z\in\partial Q^{n+1}$ and $L\in\N$, where $e=(1/2,\dots,1/2)$. It is not difficult to check that $f: A\to (J^k(\R^n),d_c)$ is a Lipschitz map. Indeed, $f$ restricted to each $\partial A_L$ is $\lip_{d_c}(F)$-Lipschitz. Next, set $c:=\max\{d_c(0, F(z)): z\in\partial Q^{n+1}\}$. If $L\not=L'$, then $$\dist(\partial A_L, \partial A_{L'})\geq 2^{\max\{L, L'\}-2}.$$ Therefore, for $x = 2^L(z-e)$ and $x'=2^{L'}(z'-e)$ with $z, z'\in\partial Q^{n+1}$, we have
\begin{equation*}
 \begin{split}
   d_c\left(f(x), f(x'))\right)  &\leq d_c(\delta_{2^L}\circ F(z), 0) + d_c(0, \delta_{2^{L'}}\circ F(z'))\\
    &\leq 2^Lc + 2^{L'}c\\
    &\leq 8c\dist(\partial A_L, \partial A_{L'})\\
    &\leq 8c \, \|x-x'\|,
 \end{split}
\end{equation*}
which proves the claim. It follows that $\iota\circ f: A\to (J^k(\R^n),d_0)$ is a Lipschitz map. Now assume that, for some $\lambda>0$, there exists a $\lambda$-Lipschitz extension $\bar{f}: \R^{n+1}\to  (J^k(\R^n), d_0)$ of $\iota\circ f$. Then, for each $L\geq 1$, the map $\bar{F}_L: Q^{n+1}\to  (J^k(\R^n), d_0)$ given by $\bar{F}_L(z):= \bar{f} (2^L(z-e))$ is a $(2^L\lambda)$-Lipschitz extension of $\iota\circ\delta_{2^L}\circ F$. By Proposition~\ref{thm:general-case}, it follows that 
\begin{equation*}
 (2^L \lambda )^{n+1} \geq 2^{L(n+k+1)}\left| \int_{Q^n}(f_0-f_1)dx\right|,
\end{equation*} 
which gives a contradiction for $L$ large enough. 
 \end{proof}

It is clear that the map $f: A\to (J^k(\R^n),d_c)$ above does not admit a Lipschitz extension $\hat f: \R^{n+1}\to  (J^k(\R^n), d_c)$ either. Indeed, otherwise $\iota\circ \hat f: \R^{n+1}\to  (J^k(\R^n), d_0)$ would be a Lip\-schitz extension of $\iota\circ f: A\to (J^k(\R^n),d_0)$ which gives a contradiction. We now prove the stronger statement Theorem~\ref{thm:main-intro}, which asserts the existence of a Lipschitz map $f: S^n\to(J^k(\R^n), d_c)$ which does not admit a Lipschitz extension.

\begin{proof}[Proof of Theorem~\ref{thm:main-intro}]
Choose any $C^{k+1}$-smooth functions  $f_0, f_1: \R^n\to\R$ satisfying \eqref{eqn:f0=f1} and
 \begin{equation*}
  \int_{Q^n} f_0dx \not=  \int_{Q^n} f_1dx.
 \end{equation*}
 We show that the Lipschitz map $F:\partial Q^{n+1}\to (J^k(\R^n), d_c)$ associated with $f_0$, $f_1$ by \eqref{eqn:def-Lip-assoc} does not admit a Lipschitz extension $F': Q^{n+1}\to (J^k(\R^n), d_c)$. Suppose, on the contrary, that there exists a $\lambda$-Lipschitz extension $F': Q^{n+1}\to (J^k(\R^n), d_c)$ of $F$ for some $\lambda>0$. Let $L>0$. Then $\iota\circ\delta_L\circ F': Q^{n+1}\to (J^k(\R^n), d_0)$ is a $(L\lambda)$-Lipschitz extension of $\iota\circ\delta_L\circ F$. By Proposition~\ref{thm:general-case}, it follows that 
\begin{equation*}
 L\lambda \geq L^{1+\frac{k}{n+1}}\left| \int_{Q^n}(f_0-f_1)dx\right|^{\frac{1}{n+1}}, 
\end{equation*} 
which gives a contradiction for $L$ large enough.
 \end{proof}

 \begin{proof}[Proof of Corollary~\ref{cor:jet-jet}]
The corollary follows from the construction in the proof of Theorem~\ref{thm:non-ext-d0} combined with the following observation. Let $E:= \{x=0\}\cap \{u^i=0: i<l\} \subset J^l(\R^m)$. Then both distances $d_0$ and $d_c$ restricted to $E$ coincide and $E$ equipped with any of these distances is isometric to the Euclidean space $\R^{d_l^m}$. Indeed let $\pi:J^l(\R^m) \rightarrow \R^{d_l^m}$ be defined by $\pi(x,u^l,\dots,u^0) := u^l$ and $s:\R^{d_l^m} \rightarrow J^l(\R^m)$ be defined by $s(u^l):=(0,u^l,0,\dots,0) \in E \subset  J^l(\R^m)$. Let $p, q \in E$. Then $\gamma(t):= s((1-t) \pi(p)+ t \pi(q))$, $t\in [0,1]$, is a $C^1$ horizontal curve joining $p$ and $q$ with 
\begin{equation*}
 length_{g_0} (\gamma) = \| \pi(q)- \pi(p)\|.
\end{equation*}
It follows that 
\begin{equation*}
 d_0(p,q)\leq d_c(p,q) \leq \| \pi(q)- \pi(p)\|.
\end{equation*}
On the other hand, for any $C^1$ curve $\gamma$ joining $p$ and $q$ in $J^l(\R^m)$, then $\hat \gamma := \pi \circ\gamma$ is a $C^1$ curve joining $\pi(p)$ and $\pi(q)$ with 
\begin{equation*}
 \|\pi(q)- \pi(p)\| \leq length_{\R^{d_l^m}} (\hat \gamma) \leq  length_{g_0} (\gamma),
\end{equation*}
which implies that
\begin{equation*}
  \|\pi(q)- \pi(p)\| \leq d_0(p,q).
\end{equation*}
Hence $d_0(p,q)= d_c(p,q) = \|\pi(q)- \pi(p)\|$ for all $p, q\in E$ and $s$ is an isometry from the Euclidean space $\R^{d_l^m}$ onto $(E,d_0) = (E,d_c)$.

To conclude, let $A\subset\R^{n+1}$ and $f:A\rightarrow (J^k(\R^n),d_c)$ be the Lipschitz map constructed in the proof of Theorem~\ref{thm:non-ext-d0}. Remembering that $d_l^m \geq n+1$ by hypothesis and considering $A\subset \R^{n+1} \hookrightarrow \R^{d_l^m}$ as a subset of $\R^{d_l^m}$, we set $\tilde f := f\circ \pi : s(A) \rightarrow J^k(\R^n)$. This map is $(d_0,d_c)$-Lipschitz. Assume that there exists an extension $\bar f:J^l(\R^m) \rightarrow J^k(\R^n)$ of $\tilde f$ which is $(d_c,d_0)$-Lipschitz. Then the map $\bar f \circ s\lfloor_{\R^{n+1}}: \R^{n+1} \rightarrow (J^k(\R^n),d_0)$ is a Lipschitz extension of $\iota\circ f$ which gives a contradiction.
\end{proof}

\begin{rmk}
 Note that it clearly follows from the preceding proofs that Theorem~\ref{thm:non-ext-d0} and Theorem~\ref{thm:main-intro} hold more generally true in a Carnot group $G$ as soon as one can find a Lipschitz map $F:\partial Q^{n+1}\to (G, d_c)$ such that 
\begin{equation*}
 \sup_{L>0} \frac{1}{L} \inf\left\{\lip_{d_0}(\bar F_L)\;|\; \bar F_L: Q^{n+1}\to (G, d_0) \text{ Lipschitz extension of } \iota \circ \delta_L \circ F\right\} =+\infty.
\end{equation*}
\end{rmk}

\section{Lower bounds for filling volume functions}\label{section:lower-bounds-FV}

As mentioned in the introduction, our construction in Section~\ref{section:main-results-proof} can be used to prove lower bounds for the $(n+1)$-st filling volume function $\FV_{n+1}$ on $(J^k(\R^n), d_0)$. Such bounds have recently been obtained by Young in \cite{Young-nilpotent-isop} with different methods.

In the following we will work with the theory of integral currents in $(J^k(\R^n), d_0)$ but we could just as well work for example with singular Lipschitz chains, as only the definition of mass and of boundary will be used in the proof of the theorem below.
For details concerning integral currents we refer to \cite{Federer-Fleming,Ambrosio-Kirchheim-currents}.
Recall that the mass $\mass(T)$ of an integral current $T$ in $(J^k(\R^n), d_0)$ is given by $$\mass(T):= \sup\{ |T(\omega)| : \text{$\omega$ compactly supported differential form with $\|\omega\|_{g_0}\leq 1$}\}$$ and that the boundary $\partial T$ is defined by $\partial T(\alpha):= T(d\alpha)$.
Now, the $(n+1)$-st filling function $\FV_{n+1}$ in $(J^k(\R^n), d_0)$ is defined by
\begin{equation*}
 \FV_{n+1}(r):= \sup\{\fillvol_{n+1}(T)\;|\; \text{ $T$ integral $n$-current with $\partial T=0$ and $\mass(T)\leq r$}\},
\end{equation*}
where $\fillvol_{n+1}(T)$ is the least mass of an integral $(n+1)$-current $S$ in $(J^k(\R^n), d_0)$ with boundary $T$. Our theorem is:

\begin{thm}\label{thm:fillvol-growth}
 For all $n,k\geq 1$, there exists $\delta>0$ such that 
 \begin{equation*}
  \FV_{n+1}(r) \geq \delta r^{\frac{n+k+1}{n}}\quad\text{ for all $r\geq 0$,}
 \end{equation*}
where $\FV_{n+1}(r)$ is the filling volume function in  $(J^k(\R^n), d_0)$ defined above.
\end{thm}

It is possible to give an explicit (though not optimal) value for $\delta$ for all $n,k\geq 1$, see the proof below.

\begin{proof}
  Choose $C^{k+1}$-smooth functions  $f_0, f_1: \R^n\to\R$ satisfying \eqref{eqn:f0=f1} and
 \begin{equation*}
  \int_{Q^n} f_0dx \not=  \int_{Q^n} f_1dx
 \end{equation*}
 and let $F$ be the Lipschitz map associated with $f_0$, $f_1$ by \eqref{eqn:def-Lip-assoc}.  Fix $L>0$ and set $\hat{F}:= \iota\circ\delta_L\circ F$. Then $T_L:=  \hat{F}_\#(\partial \Lbrack 1_{Q^{n+1}}\Rbrack)$ defines an integral $n$-current in  $(J^k(\R^n), d_0)$ and satisfies $\partial T_L = 0$. Here, $\Lbrack 1_{Q^{n+1}}\Rbrack$ is the current obtained by integrating a differential $(n+1)$-form on $Q^{n+1}$. Define a differential $n$-form on  $J^k(\R^n)$ by $\alpha:= x_1dx_2\wedge\dots\wedge dx_n\wedge du^0$ and note that $d\alpha = \omega$, where $\omega$ is as in Lemma~\ref{lem:diff-form-norm}. We obtain
 \begin{equation*}
  \hat{F}^*\alpha = h_1dh_2\wedge\dots\wedge dh_{n+1},
 \end{equation*}
 where $h_i$ denote the $x_i$-coordinate and $h_{n+1}$ the $u^0$-coordinate of $\hat{F}$. 
 Hence, as in the proof of Proposition~\ref{thm:general-case}, 
 \begin{equation*}
  T_L(\alpha) =  \int_{\partial Q^{n+1}} \hat{F}^*\alpha = (-1)^nL^{n+k+1} \int_{Q^n}(f_0-f_1)dx,
 \end{equation*}
 and in particular, $T_L\not=0$.
 Furthermore, if $S$ is an integral $(n+1)$-current in $(J^k(\R^n), d_0)$ with $\partial S = T_L$ then, together with Lemma~\ref{lem:diff-form-norm}, we get 
 \begin{equation*}
  \mass(S) \geq |S(\omega)| = |T_L(\alpha)| = L^{n+k+1}\left| \int_{Q^n}(f_0-f_1)dx\right|.
 \end{equation*}
Since
 \begin{equation*}
  \mass(T_L) \leq \lip_{d_0}(\hat{F})^n \vol(\partial Q^{n+1})\leq  L^n \lip_{d_0}(F)^n\vol(\partial Q^{n+1})
 \end{equation*}
 and since $S$ was arbitrary, we conclude
 \begin{equation*}
  \fillvol_{n+1}(T_L) \geq \delta\mass(T_L)^{\frac{n+k+1}{n}},
 \end{equation*}
 where $$\delta:= \lip_{d_0}(F)^{-(n+k+1)} \vol(\partial Q^{n+1})^{-\frac{n+k+1}{n}}\left| \int_{Q^n}(f_0-f_1)dx\right|.$$
 The proof is now complete after noting that for every $r>0$ there exists $L>0$ such that $T_L$ has mass exactly $r$.
\end{proof}

\medskip
\noindent{\bf Acknowledgment.} This paper was written while the second author was visiting the Universit\'e de Nice Sophia-Antipolis. He gratefully acknowledges the hospitality he enjoyed there.


\end{document}